\documentclass[a4paper]{article}

\usepackage[top=1.2in, bottom=1.5in, left=1.1in, right=1.8in]{geometry}
\usepackage{amsfonts}
\usepackage{amsmath}
\usepackage{amsthm,amssymb}
\usepackage{CJK,graphicx}
\usepackage{amscd}
\usepackage{amssymb}
\usepackage{mathrsfs}
\usepackage[all,cmtip]{xy}
\usepackage{lmodern}
\usepackage[Symbol]{upgreek}
\usepackage{bm}
\usepackage{eucal}
\usepackage[nopar]{lipsum}
\usepackage{rotating}
\usepackage{tikz}
\usepackage{calc}
\usepackage{tikz-cd}
\usepackage{mathabx}
\usepackage{tikz-3dplot}
\usepackage{hyperref}
\usetikzlibrary{calc}
\usetikzlibrary{decorations.markings,arrows}
\usetikzlibrary{shapes.geometric}
\usepackage{lipsum}
\usepackage{makecell}
\usepackage{times}

\newsavebox\CBox
\newcommand\hcancel[2][0.5pt]{%
	\ifmmode\sbox\CBox{$#2$}\else\sbox\CBox{#2}\fi%
	\makebox[0pt][l]{\usebox\CBox}%
	\rule[0.5\ht\CBox-#1/2]{\wd\CBox}{#1}}

\newtheorem{theorem}{Theorem}
\newtheorem{corollary}[theorem]{Corollary}

\newtheorem{lemma}[theorem]{Lemma}
\newtheorem{proposition}[theorem]{Proposition}

\newtheorem{remark}[theorem]{Remark}
\newtheorem{example}[theorem]{Example}

\numberwithin{equation}{section}

\begin{document}
	
\title{\textbf{Abundance of affine algebraic structures on stabilized cotangent bundles of surfaces}}\author{Yin Li}\date{}\maketitle

\begin{abstract}
We show that for any $g\geq2$, there exist uncountably many non-isomorphic smooth affine threefolds that are Stein deformation equivalent to $T^\ast\Sigma_g\times\mathbb{C}$, the product of the cotangent bundle of a genus $g$ oriented surface with the complex plane, answering a question of Ivan Smith. In fact, we prove that our family of smooth affine threefolds is parametrized by a $(6g-7)$-dimensional complex orbifold.
\end{abstract}

\section{Introduction}\label{section:intro}

For any $g\geq2$, denote by $\Sigma_g$ a closed oriented surface of genus $g$, and by $C_g$ a smooth genus $g$ complex algebraic curve. The difference is that for $\Sigma_g$ we only consider its smooth structure, while the curve $C_g$ can vary in a $(3g-3)$-dimensional moduli space $\mathcal{M}_g$. Let $\mathrm{Aut}(C_g)$ be the group of algebraic automorphisms of $C_g$, and let $\mathit{TC}_g$ be its tangent sheaf. The purpose of this paper is to prove the following theorem.

\begin{theorem}\label{theorem:main}
Let $g\in\mathbb{Z}_{\geq2}$. For any smooth algebraic curve $C_g$ and any $0\neq\alpha\in H^1(C_g,\mathit{TC}_g)$, there is a smooth affine threefold $X_{g,\alpha}$ that is Stein deformation equivalent to the stabilized cotangent bundle $T^\ast\Sigma_g\times\mathbb{C}$, equipped with the standard product Stein structure.

Moreover, different choices of $C_g$ in the moduli space of smooth curves give rise to non-isomorphic affine threefolds $X_{g,\alpha}$, regardless of the choice of $\alpha$. For a fixed curve $C_g$ and different choices of non-trivial first order deformations $\alpha,\beta\in H^1(C_g,\mathit{TC}_g)$, $X_{g,\alpha}$ is isomorphic to $X_{g,\beta}$ as smooth affine threefolds if and only if $\alpha$ and $\beta$ lie in the same $\mathrm{Aut}(C_g)$-orbit of the projective space $\mathbb{P}\left(H^1(C_g,\mathit{TC}_g)\right)$.
\end{theorem}

Since $\mathrm{Aut}(C_g)$ is finite for every curve $C_g$ with $g\geq2$, our theorem implies the existence of a $(6g-7)$-dimensional complex orbifold $\mathbb{P}(T\mathcal{M}_g)$, i.e. the total space of the projectivized tangent bundle of the moduli stack of smooth genus $g$ curves, that embeds in the coarse moduli space of affine algebraic structures on $T^\ast\Sigma_g\times\mathbb{C}$ with its standard Stein homotopy type.

\begin{corollary}\label{corollary:uncoun}
For any $g\geq2$, there exist uncountably many non-isomorphic smooth affine threefolds that are Stein deformation equivalent to $T^\ast\Sigma_g\times\mathbb{C}$.
\end{corollary}

It is well-known that $T^\ast S^2$ and $T^\ast T^2$ are affine surfaces in $\mathbb{C}^3$. For $g\geq2$, whether $T^\ast\Sigma_g\times\mathbb{C}$ is symplectically equivalent to a smooth affine threefold less clear due to the works of McLean \cite{mmt} and Totaro \cite{bta}, who proved that $T^\ast\Sigma_g$ is not symplectomorphic to any smooth affine surface. In fact, Totaro also proved that there is no orientation-preserving diffeomorphism from $T^\ast\Sigma_g$ to any smooth affine surface. In this sense, the conclusion of Corollary \ref{corollary:uncoun} is not trivial. However, the phenomenon that the same Stein homotopy class can support uncountably many non-isomorphic affine algebraic structures is not new: for any $n\geq2$, Jelonek \cite{zjs} constructed an $n$-dimensional Stein manifold $X$ which has uncountably many different structures as affine varieties. In his case, the spaces parametrizing the affine algebraic structures on $X$ are sometines finite group quotients of Jacobian varieties. In fact, our construction is in some sense ``orthogonal" to Jelonek's construction: we exploited the flexibility of varying the Stein structure in the same Stein homotopy class to produce uncountably many non-isomorphic affine algebraic structures, while his construction was done for a fixed Stein structure. One can therefore try to combine these two ``mutually orthogonal" constructions to get a larger family inside the coarse moduli space of affine algebraic structures on $T^\ast\Sigma_g\times\mathbb{C}$, see our discussions in Section \ref{section:var}. In particular, the family of affine threefolds $X_{g,\alpha}$ in Theorem \ref{theorem:main} is far from exhaustive as affine threefolds Stein deformation equivalent to $T^\ast\Sigma_g\times\mathbb{C}$.

On the other hand, it is not true that any Stein threefold of the form $M\times\mathbb{C}$, where $M$ is a (non-affine) Stein surface, is symplectically equivalent to a smooth affine $3$-fold. To see this, take a $5$-cycle $\Gamma_5$ and consider the associated Artin group
\begin{equation}
A_{\Gamma_5}=\left\langle x_1,\cdots,x_5|[x_i,x_{i+1}]=1\textrm{ for }i\in\mathbb{Z}_5\right\rangle.
\end{equation}
Attaching five $1$-handles to the closed ball we arrive at a $4$-manifold with boundary diffeomorphic to the $5$-fold connected sum $\#^5(S^1\times S^2)$. Picking a link $\Lambda\subset\#^5(S^1\times S^2)$ with five connected components which represent the commutators $[x_i,x_{i+1}]$ $i\in\mathbb{Z}_5$, respectively, attaching $2$-handles to $\Lambda$ we arrive at a $4$-manifold $M$ with $\pi_1(M)=\pi_1(M\times\mathbb{C})=A_{\Gamma_5}$. By Gompf's characterization of Stein surfaces \cite{rgh}, $M$ can be taken to be Stein. However, it is proved in \cite{dpst}, Theorem 11.7 that the Artin group $A_{\Gamma_5}$ cannot be the fundamental group of any quasi-affine variety since $\Gamma_5$ is not a multipartite graph, it then follows that $M\times\mathbb{C}$ cannot be diffeomorphic to any smooth affine threefold.

We end the introduction by sketching the proof of Theorem \ref{theorem:main}. For the Stein deformation equivalence between $X_{g,\alpha}$ and $T^\ast\Sigma_g\times\mathbb{C}$, the idea of the proof is simple and well-known: $h$-principle is available in this dimension to provide homotopies between Stein structures. See Theorem \ref{theorem:h} below for the precise statement. As long as one can find an affine threefold $X_{g,\alpha}$ that is diffeomorphic to $T^\ast\Sigma_g\times\mathbb{C}$ as an oriented $6$-manifold, with $c_1(X_{g,\alpha})=0$ and the induced Stein structure is subcritical, the Stein deformation equivalence will follow from the $h$-principle. In Section \ref{section:construction}, we construct the affine threefolds $X_{g,\alpha}$ and show that they have vanishing first Chern class. In Section \ref{section:ACS} we explain why the vanishing of $c_1(X_{g,\alpha})$ implies the homotopy between the complex structures on $X_{g,\alpha}$ and $T^\ast\Sigma_g\times\mathbb{C}$ in the space of almost complex structures on the underlying oriented $6$-manifold. In Section \ref{section:h}, we show the Stein structure on $X_{g,\alpha}$ is subcritical so that $h$-principle is applicable to prove the existence of a Stein homotopy. The analysis of the dependence of the algebraic structures of $X_{g,\alpha}$ on the choices of $\alpha\in H^1(C,\mathit{TC})$ is the most technical part of this paper, and this will be achieved using results from affine algebraic geometry, especially the Makar--Limanov invariant in Section \ref{section:dependence}.

\section*{Acknowledgements}
Whether $T^\ast\Sigma_g\times\mathbb{C}$ are smooth affine varieties for $g\geq2$ was asked by Ivan Smith during my visit at Cambridge University in October, 2019, therefore the question is attributed to him. He also informed me about the paper of Totaro \cite{bta}. The construction of the affine threefolds $X_{g,\alpha}$ in Section \ref{section:construction} and their dependence on the choice of $\alpha$ in Section \ref{section:dependence}, specifically the proof of Lemma \ref{lemma:ML} are accomplished with the help of GPT-5.6 Sol.

\section{Construction of the affine threefolds}\label{section:construction}

For a fixed algebraic curve $C$, pick any $0\neq\alpha\in H^1(C,\mathit{TC})$, there is a short exact sequence
\begin{equation}
0\rightarrow\mathit{TC}\rightarrow E_\alpha\xrightarrow{p}\mathcal{O}_{C}\rightarrow0,
\end{equation}
where $E_\alpha$ is a rank $2$ bundle over $C$. Consider the $\mathit{TC}$-torsor
\begin{equation}
S_\alpha:=p^{-1}(1).
\end{equation}

\begin{proposition}
When the genus of the curve $g(C)\geq2$, $S_\alpha$ is a smooth affine surface.
\end{proposition}
\begin{proof}
Smoothness follows immediately from the definition. To see that $S_\alpha$ is affine, we apply the slope criterion for the affineness of torsors in \cite{hbt}, Proposition 2.1, (ii). In our case the slope of the tangent sheaf $\mathit{TC}$ is given by $2-2g<0$, therefore the assumption $g(C)\geq2$ ensures that $S_\alpha$ is an affine surface.
\end{proof}

\begin{example}
Take $C\subset\mathbb{P}^2$ to be the Fermat quartic defined by the equation $x^4+y^4+z^4=0$, where $[x:y:z]$ are homogeneous coordinates on $\mathbb{P}^2$. There is a short exact sequence
\begin{equation}
0\rightarrow\mathcal{O}_{C}(-1)\xrightarrow{i}\mathcal{O}_{C}^{\oplus2}\xrightarrow{\pi}\mathcal{O}_{C}(1)\rightarrow0,
\end{equation}
where $i(s)=(-ys,xs)$ and $\pi(u,v)=xu+yv$. Denote by $\delta$ the connecting map in the corresponding long exact sequence, then a natural choice of $\alpha$ is given by
\begin{equation}
\alpha=\delta(-z)\in H^1\left(C,\mathcal{O}_{C}(-1)\right)=H^1(C,\mathit{TC}).
\end{equation}
The fact that $\alpha\neq0$ follows from the observation that $z$ does not lie in the image of $H^0\left(C,\mathcal{O}_{C}^{\oplus2}\right)\rightarrow H^0\left(C,\mathcal{O}_{C}(1)\right)$, which is spanned by $x$ and $y$. The extension $E_\alpha$ can be identified with the fiber product $\mathcal{O}_{C}^{\oplus2}\times_{\mathcal{O}_{C}(1)}\mathcal{O}_{C}$. Over an open subset $U\subset C$, a section of $E_\alpha$ can be written as a triple $(u,v,t)$ with $xu+yv+zt=0$, and the map $p:E_\alpha\rightarrow\mathcal{O}_{C}$ is given by the projection to $t$, therefore the $\mathit{TC}$-torsor is
\begin{equation}
S_\alpha=\left.\left\{([x:y:z],u,v)\in C\times\mathbb{C}^2\right\vert xu+yv+z=0\right\}.
\end{equation}
It is then clear that $S_\alpha$ can be embedded as an affine surface in $\mathbb{C}^5$.
\end{example}

Define
\begin{equation}
X_{\alpha}:=\mathit{Tot}(K_{S_\alpha})
\end{equation}
to be the total space of the canonical bundle over the affine surface $S_\alpha$. Since the total space of any algebraic vector bundle over an affine variety is affine, we conclude that $X_\alpha$ is an affine threefold. To indicate its independence on the genus $g$ of the algebraic curve $C$ fixed at the beginning of the construction, we shall also write it as $X_{g,\alpha}$.

We establish some basic properties of the affine threefolds $X_{g,\alpha}$ that will be used for later purposes.

\begin{lemma}\label{lemma:K}
$K_{X_{g,\alpha}}\cong\mathcal{O}_{X_{g,\alpha}}$.
\end{lemma}
\begin{proof}
Denote by $\pi:X_{g,\alpha}\rightarrow S_\alpha$ the bundle projection, we have an isomorphism $T_{X_{g,\alpha}/S_\alpha}\cong\pi^\ast K_{S_\alpha}$, where $T_{X_{g,\alpha}/S_\alpha}$ is the relative tangent bundle. Thus
\begin{equation}
K_{X_{g,\alpha}}\cong\pi^\ast(K_{S_\alpha}\otimes K_{S_\alpha}^{-1})\cong\mathcal{O}_{X_{g,\alpha}}.
\end{equation}
\end{proof}

\begin{lemma}\label{lemma:diff}
There is an orientation-preserving diffeomorphism $X_{g,\alpha}\cong T^\ast\Sigma_g\times\mathbb{C}$.
\end{lemma}
\begin{proof}
As an affine vector bundle, the torsor $q:S_\alpha\rightarrow C$ admits a smooth section, so $S_\alpha$ is diffeomorphic to the total space of the tangent bundle $\mathit{TC}$. Observe that the relative tangent sequence for $q:S_\alpha\rightarrow C$ is given by
\begin{equation}
0\rightarrow q^\ast\mathit{TC}\rightarrow\mathit{TS}_\alpha\rightarrow q^\ast\mathit{TC}\rightarrow0,
\end{equation}
where we have used the isomorphism $T_{S_\alpha/C}\cong q^\ast\mathit{TC}$, from which one obtains the isomorphism $K_{S_\alpha}\cong q^\ast K_C^{\otimes2}$. It follows that we have the diffeomorphism
\begin{equation}
X_{g,\alpha}\cong\mathit{Tot}(\mathit{TC}\oplus K_C^{\otimes2}).
\end{equation}
To show that the oriented real rank $4$ bundle $\mathit{TC}\oplus K_C^{\otimes2}$ is smoothly trivial over $C$, we calculate its second Stiefel--Whitney class. By the splitting principle,
\begin{equation}
w_2(\mathit{TC}\oplus K_C^{\otimes2})=c_1(\mathit{TC})+2c_1(K_C)=c_1(K_C)\equiv0\textrm{ }(\textrm{mod }2).
\end{equation}
It follows that we have the diffeomorphisms
\begin{equation}\label{eq:diff}
X_{g,\alpha}\cong\Sigma_g\times\mathbb{R}^4\cong T^\ast\Sigma_g\times\mathbb{C},
\end{equation}
where the second diffeomorphism follows from the fact that $\Sigma_g$ is stably parallelizable. Since both of the diffeomorphisms in (\ref{eq:diff}) are orientation-preserving, so is their composition.
\end{proof}

\begin{remark}
From the diffeomorphism $S_\alpha\cong\mathit{Tot}(\mathit{TC})$, it is easy to see that there exists an orientation-reversing diffeomorphism between the affine surface $S_\alpha$ and $T^\ast\Sigma_g$. Note that this does not contradict the result of Totaro \cite{bta}. The additional $\mathbb{C}$-factor will then allow us to erase the orientation discrepancy and obtain an oriented diffeomorphism between $S_\alpha\times\mathbb{C}$ and $T^\ast\Sigma_g\times\mathbb{C}$. However, in order to prove that the Stein threefold $T\Sigma_g\times\mathbb{C}$ is symplectically equivalent to a smooth affine variety, we cannot use $S_\alpha\times\mathbb{C}$ in place of the affine threefold $X_{g,\alpha}$: the key difference is that $S_\alpha\times\mathbb{C}$ does not have vanishing first Chern class, since $c_1(S_\alpha\times\mathbb{C})=c_1(S_\alpha)=2q^\ast c_1(C)\neq0$. The vanishing of the first Chern class will be crucial for proving the homotopy between almost complex structures in the next section.
\end{remark}

\section{The space of almost complex structures}\label{section:ACS}

Let $Y\cong\Sigma_g\times\mathbb{R}^4$ be the underlying smooth oriented manifold of both the affine threefold $X_{g,\alpha}$ constructed in Section \ref{section:construction} and the Stein threefold $T^\ast\Sigma_g\times\mathbb{C}$. Denote by $J_0$ the complex structure on $Y$ coming from the standard product Stein structure on $T^\ast\Sigma_g\times\mathbb{C}$, and by $J_1$ the complex structure defined by pushing forward the complex sturcture $J_X$ on $X_{g,\alpha}$ under the diffeomorphism $\Phi:X_{g,\alpha}\rightarrow Y$ established in Lemma \ref{lemma:diff}, i.e. $J_1=\Phi_\ast J_X$. In this section we prove the following.

\begin{proposition}\label{proposition:h}
$J_0$ and $J_1$ are homotopic as almost complex structures on $Y$.
\end{proposition}

Denote by $\mathcal{J}(Y)$ the space of all almost complex structures on $Y$ inducing the given orientation. Fix an auxiliary Riemannian metric on $Y$, which allows us to talk about the space of all orthogonal almost complex structures on $Y$ inducing the given orientation, which we denote by $\mathcal{J}^\perp(Y)$. It is well-known that the natural inclusion $\mathcal{J}^\perp(Y)\subset\mathcal{J}(Y)$ is a homotopy equivalence, see for example \cite{gmt}. The latter space has the advantage that it is the space of sections of the so-called \textit{positive twistor bundle}
\begin{equation}
Z_+(Y):=\mathit{Fr}(\mathit{TY})\otimes_{\mathit{SO}(6)}\mathit{SO}(6)/U(3),
\end{equation}
where $\mathit{Fr}(\mathit{TY})$ is the oriented orthonormal frame bundle of $\mathit{TY}$. Since $\mathit{SO}(6)/U(3)\cong\mathbb{P}^3$, $Z_+(Y)$ is a $\mathbb{P}^3$-bundle over $Y$.

Using the homotopy equivalence $\mathcal{J}^\perp(Y)\simeq\mathcal{J}(Y)$, we may assume that the complex structures $J_0$ and $J_1$ are sections of $Z_+(Y)$. In our case, since $Y$ is parallelizable, the twistor bundle $Z_+(Y)$ is actually trivial, therefore $J_0,J_1$ are represented by smooth maps $f_0,f_1:Y\rightarrow\mathbb{P}^3$. On the other hand, since $Y$ deformation retracts to the surface $\Sigma_g$, we only need to analyze maps $\Sigma_g\rightarrow\mathbb{P}^3$. Since $\mathbb{P}^3$ is simply-connected, the only relevant obstruction for the homotopy equivalence between $f_0$ and $f_1$ is given by $\pi_2(\mathbb{P}^3)\cong\mathbb{Z}$. To prove that $f_0$ and $f_1$ are homotopy equivalent, it is therefore enough to prove the vanishing of the cohomology class $a(J_0,J_1)\in H^2(Y;\mathbb{Z})$ defined by
\begin{equation}
a(J_0,J_1):=f_1^\ast h-f_0^\ast h
\end{equation}
for some (positive) generator $h\in H^2(\mathbb{P}^3;\mathbb{Z})$. 

\begin{lemma}\label{lemma:ob}
Let $c_1(Y,J)$ be the first Chern class of $Y$ with respect to the complex structure $J$, then
\begin{equation}
c_1(Y,J_1)-c_1(Y,J_0)=2a(J_0,J_1).
\end{equation}
\end{lemma}
\begin{proof}
This follows essentially from \cite{gmt}, Lemma 5.8 and the discussions right after it. More precisely, for oriented $6$-manifolds there is a $7$-equivalence $B\kappa:BU(3)\rightarrow B\mathit{Spin}^c(6)$ on the classifying spaces which identifies the set of isomorphism classes of almost complex structures on $Y$ with the set of isomorphism classes of $\mathit{Spin}^c$ structures on $Y$. Then $2a(J_0,J_1)$ equals the difference of the Chern classes of the line bundle classified by the projection $B\delta:B\mathit{Spin}^c(6)\rightarrow BS^1$. By the commutativity of the diagram
\begin{equation}
	\begin{tikzcd}
	BU(3) \arrow[d,"B\det"'] \arrow[r,"B\kappa"] &B\mathit{Spin}^c(6) \arrow[d,"B\delta"] \\
	BS^1 \arrow[r,"="] &BS^1
	\end{tikzcd}
\end{equation}
this line bundle can be identified with the line bundle classified by $B\det:BU(3)\rightarrow BS^1$, and the lemma follows.
\end{proof}

Proposition \ref{proposition:h} now follows from Lemmas \ref{lemma:K} and \ref{lemma:ob}. Note that since $c_1(T^\ast\Sigma_g)=0$, the first Chern class on $T^\ast\Sigma_g\times\mathbb{C}$ clearly vanishes.

\section{Applying the $h$-principle}\label{section:h}

To complete the proof of existence part of Theorem \ref{theorem:main}, namely the affine threefold $X_{g,\alpha}$ constructed in Section \ref{section:construction} is Stein deformation equivalent to $T^\ast\Sigma_g\times\mathbb{C}$, we use the following well-known theorem, which is a consequence of the $h$-principle. Recall that a Stein structure on an even dimensional manifold $W$ is a couple $(J,\phi)$, where $J$ is a complex structure and $\phi:W\rightarrow\mathbb{R}$ is a smooth function that is both exhausting and plurisubharmonic.

\begin{theorem}[\cite{cef}, Theorem 15.14]\label{theorem:h}
Let $(J_0,\phi_0)$ and $(J_1,\phi_1)$ be two flexible Stein structures on a manifold $W$ of dimension $2n>4$. If $J_0$ and $J_1$ are homotopic as almost complex structures, then $(J_0,\phi_0)$ and $(J_1,\phi_1)$ are homotopic as Stein structures.
\end{theorem}

To apply Theorem \ref{theorem:h}, we need to show that both of the Stein structures on $X_{g,\alpha}$ and $T^\ast\Sigma_g\times\mathbb{C}$ are flexible. Since $T^\ast\Sigma_g\times\mathbb{C}$ is subcritical, it is clearly flexible. It remains to show that $X_{g,\alpha}$ with its Stein structure coming from the embedding into the affine space is also flexible. To see this, choose a plurisubharmonic function $\psi:S_\alpha\rightarrow\mathbb{R}$ on the $\mathit{TC}$-torsor. Since $S_\alpha$ is an affine surface, it is a consequence of the Serre--Swan theorem that the canonical bundle $K_{S_\alpha}$ is a direct summand of a finite-rank trivial algebraic bundle. This gives a holomorphic embedding
\begin{equation}
\iota:X_{g,\alpha}\hookrightarrow S_\alpha\times\mathbb{C}^N.
\end{equation}
Define the function $\Psi:X_{g,\alpha}\rightarrow\mathbb{R}$ by
\begin{equation}
\Psi(s,v)=\psi(s)+\vert\!\vert\iota(s,v)\vert\!\vert^2,
\end{equation}
where $s\in S_\alpha$, $v\in\mathbb{C}^N$ is a vector in the fiber direction of $K_{S_\alpha}$ under the embedding $\iota$, and $\vert\!\vert\cdot\vert\!\vert$ denotes the standard norm on $\mathbb{C}^N$. As the restriction of the function
\begin{equation}
(s,z)\mapsto\psi(s)+\vert\!\vert z\vert\!\vert^2
\end{equation}
on $S_\alpha\times\mathbb{C}^N$, $\Psi$ is a proper plurisubharmonic function on $X_{g,\alpha}$. Since the derivative of $\Psi(s,v)$ in the fiber direction vanishes only when $v=0$, the critical points of $\Psi$ are precisely
\begin{equation}
(s,0)\textrm{ with }d\psi(s)=0.
\end{equation}
At these points, the fiberwise Hessian is positive definite, therefore the Morse indices satisfy
\begin{equation}
\mathrm{ind}_{(s,0)}\Psi=\mathrm{ind}_s\psi\leq2.
\end{equation}
Since one can take $\psi$ to be the plurisubharmonic function induced from the embedding of $S_\alpha$ into some affine space, we see that the natural Stein structure on the affine threefold $X_{g,\alpha}$ is also subcritical. We have therefore proved

\begin{proposition}
For any choice of the algebraic curve $C$ and the non-trivial first order deformation $\alpha\in H^1(C,\mathit{TC})$, the affine threefolds $X_{g,\alpha}$ are Stein deformation equivalent to $T^\ast\Sigma_g\times\mathbb{C}$.
\end{proposition}

Compare the proof of \cite{cml}, Theorem 4.9, where the flexibility of an affine threefold is proved through a less elementary argument.

\section{Dependence on the choices}\label{section:dependence}

In this section, we investigate the dependence of the affine threefolds $X_{g,\alpha}$ on the choices of the algebraic curve $C$ and the first order deformation $\alpha\in H^1(C,\mathit{TC})$ made in our construction in Section \ref{section:construction}.

We first consider the situation of varying the underlying algebraic curve $C$ in the moduli space $\mathcal{M}_g$ of smooth genus $g$ curves. Denote temporarily by $X_{C,\alpha}$ and $X_{C',\beta}$ the affine threefolds constructed in Section \ref{section:construction} using genus $g$ curves $C,C'$ and non-trivial choices of $\alpha\in H^1(T,\mathit{TC})$ and $\beta\in H^1(C',\mathit{TC}')$, respectively.

\begin{proposition}\label{proposition:C}
$X_{C,\alpha}$ and $X_{C',\beta}$ are isomorphic as affine algebraic threefolds only if $C$ is isomorphic to $C'$.
\end{proposition}
\begin{proof}
Consider the composition of projections
\begin{equation}\label{eq:pi}
\pi_C:X_{C,\alpha}\rightarrow S_\alpha\rightarrow C,
\end{equation}
which is Zariski locally an $\mathbb{A}^2$-bundle. If there is an isomorphism $\phi:X_{C,\alpha}\xrightarrow{\cong}X_{C',\beta}$, then the restriction of the composition $\pi_{C'}\circ\phi:X_{C,\alpha}\rightarrow C'$ to a fiber of $\pi_C$ gives a morphism $\mathbb{A}^2\rightarrow C'$, which is necessarily constant by our assumption that $g(C')\geq2$. Consequently, $\pi_{C'}\circ\phi$ descends to a morphism $\theta_{C,C'}:C\rightarrow C'$. Applying the same argument to $\phi^{-1}:X_{C',\beta}\rightarrow X_{C,\alpha}$ we get a morphism $C'\rightarrow C$ whose composition with $\theta_{C,C'}$ gives the identity.
\end{proof}

Next, we fix the curve $C$ in the construction of the affine threefolds $X_{g,\alpha}$ and consider the dependence of $X_{g,\alpha}$ on the choice of $\alpha\in H^1(T,\mathit{TC})$. Our goal is to prove the following:

\begin{proposition}\label{proposition:orbit}
For a fixed genus $g$ curve $C$, $X_{g,\alpha}$ is isomorphic to $X_{g,\beta}$ if and only if $[\alpha],[\beta]\in\mathbb{P}\left(H^1(C,\mathit{TC})\right)$ lie in the same $\mathrm{Aut}(C)$-orbit.
\end{proposition}

We start by introducing some notations. Define
\begin{equation}
A_\alpha:=H^0(X_{g,\alpha},\mathcal{O}_{X_{g,\alpha}}),\textrm{ }B_\alpha:=H^0(S_\alpha,\mathcal{O}_{S_\alpha}).
\end{equation}
Since $S_\alpha$ is affine and we have the isomorphism $K_{S_\alpha}\cong q^\ast K_C^{\otimes 2}$ noticed in Section \ref{section:construction}, one can write the coordinate ring of the total space of its canonical bundle as
\begin{equation}
A_\alpha=\bigoplus_{n\geq0}H^0(S_\alpha,q^\ast K_C^{\otimes-2n}),
\end{equation}
which a non-negatively graded ring, with its degree zero part $A_\alpha^0=B_\alpha$. Recall that the \textit{Makar--Limanov invariant} of the commutative algebra $A_\alpha$ is defined to be
\begin{equation}
\mathit{ML}(A_\alpha):=\bigcap_{D\in\mathrm{LND}(A_\alpha)}\ker D,
\end{equation}
where $\mathrm{LND}(A_\alpha)$ is the set of locally nilpotent derivations on $A_\alpha$. Notice that since $A_\alpha$ is the coordinate ring of $Y_\alpha$, a locally nilpotent derivation $D$ corresponds precisely to an algebraic action of the additive group $\mathbb{G}_a$ on $Y_\alpha$, and $\ker D$ corresponds to the ring of invariant functions under the $\mathbb{G}_a$-action. See \cite{gfa} for details.

The key point of the proof of Theorem \ref{proposition:orbit} is the following lemma, which computes the Makar--Limanov invariant of $A_\alpha$.

\begin{lemma}\label{lemma:ML}
We have $\mathit{ML}(A_\alpha)=B_\alpha$.
\end{lemma}
\begin{proof}
We first prove the easier inclusion $\mathit{ML}(A_\alpha)\subset B_\alpha$. As a line bundle, $X_{g,\alpha}=\mathit{Tot}(K_{S_\alpha})$ carries the action of $\mathbb{G}_a$ by fiberwise translations. More precisely, every $s\in H^0(S_\alpha,K_{S_\alpha})$ gives rise to a locally nilpotent derivation $D_s\in\mathrm{LND}(A_\alpha)$. See \cite{adf}, Section 2.2.1. Since the corresponding $\mathbb{G}_a$-actions associated to $s\in H^0(S_\alpha,K_{S_\alpha})$ rescales in the fiber direction of $K_{S_\alpha}$, we have
\begin{equation}
\bigcap_{s\in H^0(S_\alpha,K_{S_\alpha})}\ker D_s=B_\alpha.
\end{equation}
By definition, this implies that $\mathit{ML}(A_\alpha)\subset B_\alpha$.

To see that $\mathit{ML}(A_\alpha)\supset B_\alpha$, we argue by contradiction. Suppose there is a $\partial\in\mathrm{LND}(A_\alpha)$ that does not vanish on $B_\alpha$, by \cite{bmo}, Lemma 11 one can find a homogeneous locally nilpotent derivation $D$ on $A_\alpha$ which is non-vanishing on $B_\alpha$. Without loss of generality, assume $\deg(D)=k\in\mathbb{Z}_{\geq0}$. Since $B_\alpha$ is trivially graded, we have $D(B_\alpha)\subset A_\alpha^k$. 

Observe that $\mathbb{G}_a$ can only act non-trivially on the fiber direction of $\pi_\alpha:X_{g,\alpha}\rightarrow S_\alpha\rightarrow C$. This is because any orbit of the $\mathbb{G}_a$-action on $X_{g,\alpha}$ gives rise to a morphism $\mathbb{A}^1\rightarrow C$, which by the properness of $C$ extends to $\mathbb{P}^1\rightarrow C$, therefore must be constant with the assumption that $g(C)\geq2$. It follows that every $\mathbb{G}_a$-orbit lies in a fiber of $\pi_\alpha$. Equivalently, every associated locally nilpotent derivation annihilates the function field $K=\mathbb{C}(C)\subset\mathbb{C}(Y_\alpha)$. Note that $D\in\mathrm{LND}(A_\alpha)$ extends uniquely to a derivation on the field of fractions, and by abuse of notations we still use $D$ to denote its extension.

Over a generic point $\mathrm{Spec}K$ of $C$, we choose coordinates $x,t$, where $x$ is the coordinate on the torsor $S_\alpha$ and $t$ is the coordinate in the fiber direction of $K_{S_\alpha}$, it follows that
\begin{equation}
\deg(x)=0,\textrm{ }\deg(t)=1.
\end{equation}
Since $D$ is homogeneous of degree $k$, we have
\begin{equation}
D(t)=a(x)t^{k+1},\textrm{ }D(x)=b(x)t^k
\end{equation}
for some $a(x),b(x)\in K[x]$. Since $t$ divides $D(t)$, local nilpotency of $D$ on $A_\alpha$ forces $D(t)=0$ by degree considerations. Thus we must have $D=b(x)t^k\frac{\partial}{\partial x}$. Again by the local nilpotency of $D$, $b(x)$ must be constant, therefore
\begin{equation}\label{eq:D}
D=\mu t^k\frac{d}{dx}\textrm{ for some }\mu\in K.
\end{equation}

Since the restriction $D|_{B_\alpha}:B_\alpha\rightarrow A_\alpha^k$ corresponds geometrically to a global section of
\begin{equation}
T_{S_\alpha/C}\otimes q^\ast K_C^{\otimes-2k}\cong q^\ast\mathit{TC}\otimes q^\ast K_C^{\otimes-2k}\cong q^\ast K_C^{\otimes-(2k+1)},
\end{equation}
it follows from (\ref{eq:D}) that this section is the pullback of a rational section $\sigma$ of the line bundle $K_C^{\otimes-(2k+1)}$. Since $q^\ast\sigma$ is regular on $S_\alpha$, the faithful flatness of the projection $q:S_\alpha\rightarrow C$ implies that
\begin{equation}
\sigma\in H^0\left(C,K_C^{\otimes-(2k+1)}\right),
\end{equation}
which contradicts with the assumption that $g(C)\geq2$ if $\sigma$ is non-trivial. Thus $D\in\mathrm{LND}(A_\alpha)$ has to vanish on $B_\alpha$ and we have proved $\mathit{ML}(A_\alpha)\supset B_\alpha$.
\end{proof}

We can now finish the proof of Proposition \ref{proposition:orbit}, therefore also the proof of Theorem \ref{theorem:main}.

\begin{proof}[Proof of Proposition \ref{proposition:orbit}]
Given any isomorphism $X_{g,\alpha}\cong X_{g,\beta}$, there is an isomorphism between the corresponding Makar--Limanov invariants, i.e.
\begin{equation}
\mathit{ML}(A_\alpha)\cong\mathit{ML}(A_\beta).
\end{equation}
By Lemma \ref{lemma:ML}, it follows that $B_\alpha\cong B_\beta$, thus there is an isomorphism $\varphi:S_\alpha\xrightarrow{\cong} S_\beta$ between the corresponding torsors. Denote by $q_\alpha$ and $q_\beta$ the projections of the torsors $S_\alpha$ and $S_\beta$ to the same curve $C$, since the morphism $q_\beta\circ\varphi:S_\alpha\rightarrow C$ is constant along the $\mathbb{A}^1$-fibers for the same reason as in the proof of Proposition \ref{proposition:C}, we get a morphism $\theta_C:C\rightarrow C$, which is easily seen to be an automorphism since $\varphi$ is an isomorphism. After replacing $S_\beta$ with $h^\ast S_\beta$, we get an isomorphism between $S_\alpha$ and $S_\beta$ over $C$. In terms of the local coordinates on the $\mathbb{A}^1$-fiber over the open subset $U_i\subset C$, such an isomorphism has the form
\begin{equation}
x_i\mapsto f_ix_i+g_i
\end{equation}
for some $f_i,g_i\in\mathbb{C}$. Compatibility on overlaps of the charts $\{U_i\}$ implies that the $u_i$'s define an automorphic of $\mathit{TC}$. Since $H^0(C,\mathcal{O}_C^\times)\cong\mathbb{C}^\times$, we see that $f_i=\lambda\in\mathbb{C}^\times$ for all $i$. The translations $g_i$ can only change the class in $H^1(C,\mathit{TC})$ by a \u{C}ech coboundary. We therefore have
\begin{equation}
\alpha=\lambda\cdot\theta_C^\ast\beta\in H^1(C,\mathit{TC}).
\end{equation}
Conversely, such a relation produces an isomorphism of the torsors, and hence also an isomorphism of the total spaces of their canonical bundles.
\end{proof}

\section{A variation of the main construction}\label{section:var}

We discuss how one can easily modify the construction of Section \ref{section:construction} to produce more affine threefolds that are Stein homotopic to $T^\ast\Sigma_g\times\mathbb{C}$. Needless to say, the ideas here are motivated by the work of Jelonek \cite{zjs}.

Pick a non-trivial line bundle $\mathcal{L}\in\mathrm{Pic}^0(C)$ in the degree $0$ Picard group, and define
\begin{equation}
X_{g,\alpha}^\mathcal{L}:=\mathit{Tot}(K_{S_\alpha}\otimes q^\ast\mathcal{L}),
\end{equation}
which is clearly an affine threefold.

\begin{proposition}\label{proposition:bihol}
$X_{g,\alpha}^\mathcal{L}$ is biholomorphic to $X_{g,\alpha}$ for any choice of $\mathcal{L}$.
\end{proposition}
\begin{proof}
This is a consequence of Cartan's Theorem B, which implies that $\mathrm{Pic}_\mathrm{hol}(S_\alpha)\cong H^2(S;\mathbb{Z})$, where $\mathrm{Pic}_\mathrm{hol}(S_\alpha)$ denotes the Picard group of holomorphic line bundles on $S_\alpha$. Since $\deg(\mathcal{L})=0$, we have $c_1(q^\ast\mathcal{L})=0$. It follows that as holomorphic line bundles, the twisted canonical bundle $K_{S_\alpha}\otimes q^\ast\mathcal{L}$ is isomorphic to untwisted one $K_{S_\alpha}$.
\end{proof}

\begin{proposition}\label{proposition:noniso}
As affine algebraic threefolds, $X_{g,\alpha}^\mathcal{L}$ is not isomorphic to $X_{g,\beta}$ for any choices of the algebraic curve $C'$ and $0\neq\beta\in H^1(C',\mathit{TC}')$.
\end{proposition}
\begin{proof}
Denote by $\pi_\mathcal{L}:X_{g,\alpha}^\mathcal{L}\rightarrow S_\alpha$ the line bundle projection, we have
\begin{equation}
K_{X_{g,\alpha}^\mathcal{L}}\cong\pi_\mathcal{L}^\ast\left(K_{S_\alpha}\otimes(K_{S_\alpha}\otimes q^\ast\mathcal{L})^{-1}\right)\cong\pi_\mathcal{L}^\ast q^\ast\mathcal{L}^{-1}
\end{equation}
as algebraic line bundles. To prove the proposition, it is enough to show that the line bundle $\pi_\mathcal{L}^\ast q^\ast\mathcal{L}^{-1}$ is non-trivial on $X_{g,\alpha}^\mathcal{L}$. To see this, note that as a $\mathit{TC}$-torsor, $q:S_\alpha\rightarrow C$ is an $\mathbb{A}^1$-weak equivalence, therefore induces an isomorphism
\begin{equation}
q^\ast:\mathrm{Pic}(C)\xrightarrow{\cong}\mathrm{Pic}(S_\alpha)
\end{equation}
between Picard groups. See \cite{mva}, Proposition 3.8. Since the map
\begin{equation}
\pi_\mathcal{L}^\ast:\mathrm{Pic}(S_\alpha)\rightarrow\mathrm{Pic}(X_{g,\alpha}^\mathcal{L})
\end{equation}
induced by bundle projection is clearly injective, the non-triviality of $\pi_\mathcal{L}^\ast q^\ast\mathcal{L}^{-1}$ follows from the non-triviality of $\mathcal{L}$.
\end{proof}

We remark that Propositions \ref{proposition:bihol} and \ref{proposition:noniso} are not sufficient to give an embedding of the fiber product $\mathbb{P}(T\mathcal{M}_g)\times_{\mathcal{M}_g}\mathcal{J}_g$, where $\mathcal{J}_g\rightarrow\mathcal{M}_g$ is the universal Jacobian over the moduli space of smooth genus $g$ curves, into the coarse moduli space of affine algebraic structures on $T^\ast\Sigma_g\times\mathbb{C}$ which are homotopic to the standard Stein structure, because whether $X_{g,\alpha}^\mathcal{L}$ is isomorphic to $X_{g,\alpha}^\mathcal{K}$ for different choices of $\mathcal{L},\mathcal{K}\in\mathrm{Pic}^0(C)$ is in general not known. It remains an interesting question to produce an exhaustive family of such affine threefold structures on the stabilized cotangent bundles $T^\ast\Sigma_g\times\mathbb{C}$.

\end{document}